\documentclass[10pt]{amsart}
\usepackage{amsmath}
\usepackage[usenames,dvipsnames]{color}
\usepackage{parskip}
\usepackage{amsfonts}
\usepackage{amscd}
\usepackage[centertags]{amsmath}
\usepackage{amssymb}
\usepackage{stmaryrd}
\usepackage[all,cmtip]{xy}
\usepackage[english]{babel}
\usepackage{tabularx}
\usepackage{amsxtra}
\usepackage{euscript}
\usepackage[T1]{fontenc}
\usepackage{doc, exscale, fontenc, latexsym, syntonly}
\usepackage{amsfonts}
\usepackage{amsthm}
\usepackage{graphicx}
\usepackage{mciteplus}
\usepackage{cite}

\newtheorem{theorem}{Theorem}[section]
\newtheorem{corollary}[theorem]{Corollary}

\newtheorem{proposition}[theorem]{Proposition}

\theoremstyle{definition}
\newtheorem{definition}[theorem]{Definition}
\newtheorem{example}[theorem]{Example}
\theoremstyle{remark}

\numberwithin{figure}{section}
\numberwithin{table}{section}

\newcommand*\acknowledgment[1]{%
	\begingroup\noindent
	\rightskip\leftskip
	\begin{flushleft}\textbf{\large Acknowledgment.}\, #1%
		\par\vspace*{1mm}\end{flushleft}\endgroup}
\begin{document}

\title[Certain Topological Methods For Computing Digital Topological Complexity]{Certain Topological Methods For Computing Digital Topological Complexity}

\author{MEL\.{I}H \.{I}S and \.{I}SMET KARACA}
\date{\today}

\address{\textsc{Melih Is}
Ege University\\
Faculty of Sciences\\
Department of Mathematics\\
Izmir, Turkey}
\email{melih.is@ege.edu.tr}
\address{\textsc{Ismet Karaca}
Ege University\\
Faculty of Science\\
Department of Mathematics\\
Izmir, Turkey}
\email{ismet.karaca@ege.edu.tr}



\begin{abstract}
In this paper, we examine the relations of two closely related concepts, the digital Lusternik-Schnirelmann category and the digital higher topological complexity, with each other in digital images. For some certain digital images, we introduce $\kappa-$topological groups in the digital topological manner for having stronger ideas about the digital higher topological complexity. Our aim is to improve the understanding of the digital higher topological complexity. We present examples and counterexamples for $\kappa-$topological groups. 
\end{abstract}

\maketitle

\section{Introduction}
\label{intro}
\quad The interaction betweens two popular topics (digital image processing and robotics) can often be very valuable in science. The subject of robotics is rapidly increasing its popularity among who study on topology. In the digital topology, one of the extraordinary fields of mathematics, with using topological properties, we can melt these two topics in one pot. Thus, we are trying to build a theoretical bridge between motion planning algorithms of a robot and digital image analysis. In future studies, we think that the theoretical knowledge will focus on the applications of industry, perhaps in other fields. In more details, an autonomous robot is expected to be able to determine its own direction and route without any help. There are many types of robots using motion planning algorithms for this duty, especially industrial and mobile robots. Industrial robots undertake some tasks in various fields such as assembly and welding works in the industry. As an example of mobile robots, we can consider unmanned aerial vehicles and the room cleaning robots.

\quad Digital topology \cite{Ros:1979} has been developing and increasing its scientific importance. Many significant invariants of topology, especially homotopy, homology and cohomology, have a substantial value for digital images. The digital homotopy is, in particular, our fundamental equipment. You can easily have the comprehensive knowledge about the digital homotopy from \cite{Boxer:1999,Boxer:2005,Boxer:2006,Boxer2:2006,BoxKar:2012,BoxStae:2016,Kong:1989}. As we mentioned, we are concerned with topological interpretations of robot motions in digital images. Farber \cite{Farber:2003} studies topological complexity of motion planning. The construction of motion planning algorithms on a lot of topological spaces has been discussed. Different types of topological structures have been expertly placed in the theory of this subject \cite{Farber:2008}. What we add to these studies is related to the part of the notion of higher topological complexity $TC_{n}$ in digital images. Rudyak \cite{Rudyak:2010} first define it for ordinary topological spaces. It is integrated into digital topology by \.{I}s and Karaca \cite{KaracaIs:2018}. The digital meaning of the Lusternik-Schnirelmann category $cat$ appears in \cite{BorVer:2018} and the importance of the notion is its close relationship with $TC_{2}$, in other saying the special case of $TC_{n}$, when $n=2$. Therefore, we frequently use $cat$ having precious results about not only $TC_{2}$ but also $TC_{n}$. The definition of $cat$, $TC_{2}$ and $TC_{n}$ is expressed by the concept of the digital Schwarz genus of some digital fibrations. In a way, we can figure out that these concepts are in a close relation with all the properties of the notion digital Schwarz genus and each other. Moreover, the structure that helps us stating one of the strongest relationship between $cat$ and $TC$ is the $\kappa-$topological group of digital images, where $\kappa$ is an adjacency relation of a digital image. We introduce this notion and outline the framework of it.

\quad The structure of the paper is as follows. First of all, we start by recalling the cornerstones of digital topology that we often use in this study and some previously emphasized properties for the digital higher topological complexity. After Section Preliminaries, we give some results about the digital Schwarz genus of a digital image and the digital Lusternik-Schnirelmann category of a digital image. These results are general, i.e., not spesific for the digital higher topological complexity. However, our aim is to obtain a lower or an upper bound for specifying TC$_{n}$ in digital spaces. We deal with topological groups in digital setting in Section 4. After we give the definition of a $\kappa-$topological group of a digital image, we present interesting examples for some digital images. Before the last section, we obtain various results using $cat$ and $\kappa-$topological groups. Moreover, we give examples and counterexamples about certain digital images.

\section{Preliminaries}
\label{sec:1}
\quad For any positive integer $r$, a \textit{digital image} $(Y,\lambda)$ consists of a finite subset $Y$ of $\mathbb{Z}^{r}$ and an adjacency relation $\lambda$ for the elements of $Y$ such that the relation is defined as follows: Two distinct points $y$ and $z$ in $\mathbb{Z}^{r}$ are \textit{$c_{k}-$adjacent} \cite{Boxer:1999} for a positive integer $k$ with $1 \leq k \leq r$, if there are at most $k$ indices $i$ such that $|y_{i} - z_{i}| = 1$ and for all other indices $i$ such that $|y_{i} - z_{i}| \neq 1$, $y_{i} = z_{i}$. In the one-dimensional case, if we study in $\mathbb{Z}$, then we merely have the $2-$adjacency. There are completely two adjacency relations $4$ and $8$ in $\mathbb{Z}^{2}$ and completely three adjacency relations $6$, $18$ and $24$ in $\mathbb{Z}^{3}$.

\quad Let $Y \subset \mathbb{Z}^{r}$ be a digital image. Then $Y$ is \textit{$\lambda-$connected} \cite{Herman:1993} if and only if for any $y,z \in Y$ with $y \neq z$, there is a set $\{y_{0}, y_{1}, ..., y_{m}\}$ of points of $Y$ such that $y = y_{0}$, $z = y_{m}$ and $y_{i}$ and $y_{i+1}$ are $\lambda-$adjacent, where $i = 0, 1, ..., m-1$. Let $(Y_{1},\lambda_{1})$ and $(Y_{2},\lambda_{2})$ be two digital images in $\mathbb{Z}^{r_{1}}$ and $\mathbb{Z}^{r_{2}}$, respectively. Let $f : Y_{1} \longrightarrow Y_{2}$ be a map. Then $f$ is \textit{$(\lambda_{1},\lambda_{2})-$continuous} \cite{Boxer:1999} if, for any $\lambda_{1}-$connected subset $A_{1}$ of $Y_{1}$, $f(A_{1})$ is also $\lambda_{2}-$connected. [Proposition 2.5,\cite{Boxer:1999}] proves that the composition of any two digitally continuous maps is again digitally continuous.

\quad A digital map $f : (Y_{1},\lambda_{1}) \longrightarrow (Y_{2},\lambda_{2})$ is called a \textit{$(\lambda_{1},\lambda_{2})-$isomorphism} \cite{Boxer2:2006} if $f$ is bijective, $(\lambda_{1},\lambda_{2})-$continuous and also $f^{-1}$ is $(\lambda_{2},\lambda_{1})-$continuous. Let $[0,r]_{\mathbb{Z}}$ be a digital image with a positive integer $r$. It is clearly has $2-$\linebreak adjacency. For any digital image $(Y,\lambda)$, if $f : [0,r]_{\mathbb{Z}} \longrightarrow Y$ is a $(2,\lambda)-$continuous map with $f(0) = y_{1}$ and $f(r) = y_{2}$, then $f$ is called a \textit{digital path} \cite{Boxer:2006} between the initial point $y_{1}$ and the final point $y_{2}$. Two digital paths $f_{1}$ and $f_{2}$ in $(Y,\lambda)$ are \textit{adjacent paths} \cite{KaracaIs:2018} if, for all times $t$, they are digitally connected.

\quad Given two digital images $(Y_{1},\lambda_{1})$ and $(Y_{2},\lambda_{2})$ in $\mathbb{Z}^{r_{1}}$ and $\mathbb{Z}^{r_{2}}$, respectively such that $f_{1},f_{2} : Y_{1} \longrightarrow Y_{2}$ are two $(\lambda_{1},\lambda_{2})-$continuous maps. The maps $f_{1}$ and $f_{2}$ are \textit{ $(\lambda_{1},\lambda_{2})-$homotopic} \cite{Boxer:1999} in $Y$ (denoted by $f_{1} \simeq_{(\lambda_{1},\lambda_{2})}f_{2}$), if, for a positive integer $m$, there is a digital map $F : Y_{1} \times [0,m]_{\mathbb{Z}} \longrightarrow Y_{2}$ which admits the following conditions:
\begin{itemize}
	\item for all $y \in Y_{1}$, $F(y,0)=f_{1}(y)$ and $F(y,m) = f_{2}(y)$;
	
	\item for all $y \in Y_{1}$ and for all $s \in [0,m]_{\mathbb{Z}}$, \[F_{y}:[0,m]_{\mathbb{Z}} \longrightarrow Y_{2}\] 
	\hspace*{4.7cm} $s \longmapsto F_{y}(s)= F(y,s)$
	
	is $(2,\lambda_{2})-$continuous;
	
	\item for all $s \in [0,m]_{\mathbb{Z}}$ and for all $y \in Y_{1}$,
	\[F_{s} : Y_{1} \longrightarrow Y_{2}\] \hspace*{4.7cm}$y \longmapsto F_{s}(y)=F(y,s)$ 
	
	is $(\lambda_{1},\lambda_{2})-$continuous.
\end{itemize}
The function $F$ in the definition above is said to be \textit{digital homotopy} between $f_{1}$ and $f_{2}$. Note that a homotopy relation, in the digital sense, is equivalence on digitally continuous maps. \cite{Boxer:1999}.

\quad Let $(Y_{1},\lambda_{1})$ and $(Y_{2},\lambda_{2})$ be any digital images for which $f : Y_{1} \rightarrow \nolinebreak Y_{2}$ is digitally continuous. Then $f$ is called \textit{$(\lambda_{1},\lambda_{2})-$nullhomotopic} \cite{Boxer:1999} in $Y_{2}$ on condition that $f$ is $(\lambda_{1},\lambda_{2})-$homotopic to a constant map in $Y_{2}$. Assume now that the digital map \linebreak $f : (Y_{1},\lambda_{1}) \longrightarrow (Y_{2},\lambda_{2})$ is $(\lambda_{1},\lambda_{2})-$continuous. If there exists a $(\lambda_{2},\lambda_{1})-$continuous map $g : (Y_{2},\lambda_{2}) \longrightarrow (Y_{1},\lambda_{1})$ for which $g \circ f \simeq_{(\lambda_{1},\lambda_{1})} id_{Y_{1}}$ and $f \circ g \simeq_{(\lambda_{2},\lambda_{2})} id_{Y_{2}}$, then $f$ is a \textit{$(\lambda_{1},\lambda_{2})-$homotopy equivalence} \cite{Boxer:2005}. It is said to be that a digital image $(Y,\lambda)$ is \textit{$\lambda-$contractible} \cite{Boxer:1999} if $id_{Y}$ is $(\lambda,\lambda)-$homotopic to a map $c$ of digital images for some $c_{0} \in Y$, where $c : Y \longrightarrow Y$ is defined with $c(y) = c_{0}$ for all $y \in Y$.

\quad The adjacency relation varies in several digital images. For instance, an adjacency relation on the set of digital functions is discussed in \cite{LupOpreaScov:2019}. For any images $X$ and $Y$, a \textit{function space map in digital images} is stated with the set of all maps $X \rightarrow Y$ with adjacency as follows: for any two maps $f$, $g : X \rightarrow Y$, they are called \textit{adjacent} in the set of digital function spaces if $f(x)$ and $g(x^{'})$ are adjacent points in $Y$ whenever $x$ and $x^{'}$ are adjacent points in $X$. Another crucial example is given on the cartesian product of digital images \cite{BoxKar:2012}: Let $(Y,\lambda_{1})$ and $(Z,\lambda_{2})$ be any two digital images such that the points $(y,z)$ and $(y{'},z{'})$ belong to $Y \times Z$. Then $(y,z)$ and $(y{'},z{'})$ are \textit{adjacent in the cartesian product digital image} $Y \times Z$ if one of the following conditions holds:
\begin{itemize}
	\item $y = y{'}$ and $z = z{'}$; or
	
	\item $y = y{'}$ and $z$ and $z{'}$ are $\lambda_{2}-$adjacent; or
	
	\item $y$ and $y{'}$ are $\lambda_{1}-$adjacent and $z = z{'}$; or	
	
	\item $y$ and $y{'}$ are $\lambda_{1}-$adjacent and $z$ and $z{'}$ are $\lambda_{2}{'}-$adjacent.
\end{itemize}
We define the \textit{minimal adjacency relation for the cartesian product of digital images} as the smallest number of all possible adjacency relations for the product image. Recall that the set $\{6,18,26\}$ of adjacency relations on $\mathbb{Z}^{3}$. We must choose $6-$adjacency as the minimal adjacency relation for $\mathbb{Z}^{3}$. 

\quad If a map $p : (X,\kappa_{1}) \longrightarrow (Y,\kappa_{2})$ has the digital homotopy lifting property for every digital image, then $p$ is called a \textit{digital fibration} \cite{EgeKaraca:2017}.

\begin{definition}\cite{MelihKaraca}
	Let $(X,\kappa_{1})$ and $(Y,\kappa_{2})$ be any digitally connected images. A \textit{digital fibrational substitute} of a map $f:(X,\kappa_{1}) \longrightarrow (Y,\kappa_{2})$ is, in the digital sense,  a fibration $\widehat{f}:(Z,\kappa_{3}) \longrightarrow (Y,\kappa_{2})$ for which $\widehat{f} \circ h = f$, i.e.,
	\begin{displaymath}
	\xymatrix{
		X \ar[r]^{h} \ar[dr]_f &
		Z \ar@{.>}[d]^{\widehat{f}} \\
		& Y, }
	\end{displaymath}
	where $h$ is an equivalence in the sense of digital homotopy.
\end{definition}

\begin{definition}\cite{MelihKaraca}
	The digital Schwarz genus of a digital fibration \linebreak$p:(E,\lambda_{1}) \longrightarrow (B,\lambda_{2})$, denoted by $genus_{\lambda_{1},\lambda_{2}}(p)$, is defined as a minimum number $k$ for which $\{V_{1}, V_{2}, ..., V_{k}\}$ is a cover of $B$ with the property that there is a continuous map of digital images $s_{i}:(V_{i},\lambda_{1}) \longrightarrow (E,\lambda_{2})$ such that $p \circ s_{i} = id_{V_{i}}$ for all $1 \leq i \leq k$.
\end{definition} 

\quad In the digital meaning, we note that the Schwarz genus of a map $p$ is the Schwarz genus of the digital fibrational substitute of $p$. Moreover, the fact that the Schwarz genus of a digital map is invariant from the chosen fibrational substitute is proved in [Lemma 3.4,\cite{MelihKaraca}].
\begin{definition}\cite{MelihKaraca}
	Let $X^{[0,m]_{\mathbb{Z}}}$ be a digital function space of all continuous functions from $[0,m]_{\mathbb{Z}}$ to a digitally connected image $(X,\kappa)$ for any positive integer $m$. Then \textit{topological complexity of digital images} \[TC(X,\kappa) = genus_{\lambda_{\ast},\kappa_{\ast}}(p),\] where $p : (X^{[0,m]_{\mathbb{Z}}},\lambda_{\ast}) \longrightarrow (X \times X,\kappa_{\ast})$, $p(w) = (w(0),w(m))$ is a fibration of digital images for any $w \in X^{[0,m]_{\mathbb{Z}}}$.
\end{definition}

\begin{definition}\cite{MelihKaraca}
	Let $[0,m]_{\mathbb{Z}}^{i}$ denote the $i-$th digital interval with endpoint $m$. Given $n$ digital intervals $[0,m_{1}]_{\mathbb{Z}}^{1}$, ... , $[0,m_{n}]_{\mathbb{Z}}^{n}$ and denote $J_{n}$ with the wedge of the digital intervals for $n \geq 1$ and $n \in \mathbb{N}$, where $0_{i} \in [0,m_{i}]_{\mathbb{Z}}^{i}$, $i = 1, ... ,n$, are identified. Let $X$ be a digitally connected space. Then the \textit{higher topological complexity of digital images} \[TC_{n}(X,\kappa) = genus_{\lambda_{\ast},\kappa_{\ast}}(e_{n}),\]
	where $e_{n} : (X^{J_{n}},\lambda_{\ast}) \longrightarrow (X^{n},\kappa_{\ast})$ is a fibration of digital images for any \linebreak $e_{n}(f) = (f((m_{1})_{1}),...,f((m_{n})_{n}))$ and $(m_{i})_{i}$, for each $i$, is the endpoint of the $i-$th interval. 
\end{definition}

\quad Note that, in digital spaces, the higher topological complexity has significant rules \cite{MelihKaraca}. One of them is that $TC_{1}$ is always equal to $1$. Another is the coincidence of $TC_{2}$ with $TC$, when $n=2$. Moreover, the number $TC_{n}$ is not greater than $TC_{n+1}$ at all. 
\begin{proposition} \cite{MelihKaraca}\label{p4}
	$TC_{n}(X,\kappa) = genus_{\kappa,\kappa_{*}}d_{n}$, where $d_{n} : (X,\kappa) \rightarrow (X^{n},\kappa_{*})$ is a diagonal map such that $\kappa_{*}$ is an adjacency relation for the image $X^{n}$.
\end{proposition}

\quad We finish this section with the the digital Lusternik-Schnirelmann category $cat_{\kappa}(X)$ of the image $(X,\kappa)$. In \cite{BorVer:2018}, $X$ is covered with $(k+1)$ sets $U_{1}, U_{2}, ..., U_{k+1}$ in the definition of the digital L-S category. We note that we use $k$ sets $U_{1}, U_{2}, ..., U_{k}$ to do it.
\begin{definition} \cite{BorVer:2018}
	\textit{The digital Lusternik-Schnirelmann category} of a digital image $X$ (denoted $cat_{\kappa}(X)$) is defined to be the minimum number $k$ for which there is a cover $\{U_{1}, U_{2}, ..., U_{k}\}$ of  $X$ that satisfies each inclusion map from $U_{i}$ to $X$, for $i = 1, ..., k$, is $\kappa-$nullhomotopic in $X$.
\end{definition}
\begin{example}\label{e4}
	Let $H = \{a, b, c, d, e, f, g, h\}$ be an image in $\mathbb{Z}^{2}$ for which it has $4-$adjacency (see Figure \ref{fig:2}) such that
	\begin{eqnarray*}
		&&a=(0,-1), \hspace*{0.2cm} b=(0,0), \hspace*{0.2cm} c=(0,1), \hspace*{0.2cm} d=(1,1),\\
		&&e=(2,1), \hspace*{0.2cm} f=(2,0), \hspace*{0.2cm} g=(2,-1), \hspace*{0.2cm} h=(1,-1).
	\end{eqnarray*}
	\begin{figure*}[h]
		\centering
		\includegraphics[width=0.35\textwidth]{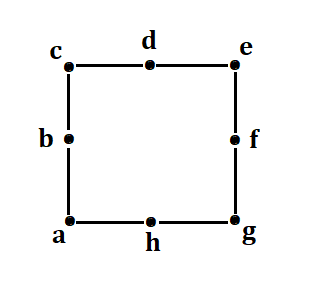}
		\caption{The Digital Image $H$.}
		\label{fig:2}
	\end{figure*}
	Since $H$ is not $4-$contractible, $cat_{4}(H) > 1$. We prove that $cat_{4}(H) = 2$. Let $M_{1} = \{b, c, d, e\}$ and $M_{2} = \{a,f,g,h\}$. Then $H = M_{1} \cup M_{2}$. We set the digital homotopy
	\[F_{1} : M_{1} \times [0,3]_{\mathbb{Z}} \rightarrow H\]
	\hspace*{5.3cm} $(m,s) \longmapsto F_{1}(m,s) =  \begin{cases}
	i_{1}(m), & s=0 \\
	i_{1}(m), & s=1, m=c,d,e\\
	c, & s=1, m=b\\
	i_{1}(m), &s=2, m=d,e\\
	d, & s=2, m=b,c\\
	e, & s=3,
	\end{cases}$
	where $i_{1} : M_{1} \rightarrow H$ is a digital inclusion map. Hence, $i_{1}$ is digitally nullhomotopic. Similarly, the digital homotopy 
	\[F_{2} : M_{2} \times [0,3]_{\mathbb{Z}} \rightarrow H\]
	\hspace*{5.3cm} $(m,s) \longmapsto F_{2}(m,s) =  \begin{cases}
	i_{2}(m), & s=0 \\
	i_{2}(m), & s=1, m=a,g,h\\
	g, & s=1, m=f\\
	i_{2}(m), &s=2, m=a,h\\
	h, & s=2, m=f,g\\
	a, & s=3,
	\end{cases}$
	where $i_{2} : M_{2} \rightarrow H$ is a digital inclusion map, shows that $i_{2}$ is digitally nullhomotopic (See Figure \ref{fig:1}). This shows that $cat_{4}(H) = 2$.
	\begin{figure*}[h]
		\centering
		\includegraphics[width=0.75\textwidth]{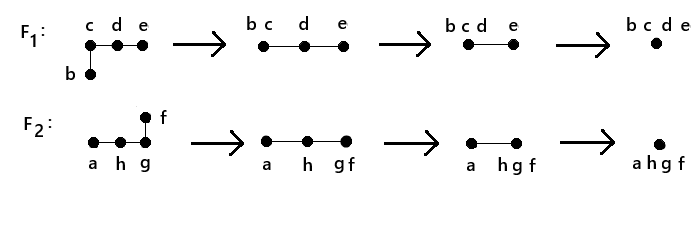}
		\caption{The Digital Homotopies $F_{1}$ and $F_{2}$ at each time $s$.}
		\label{fig:1}
	\end{figure*}
\end{example}

\begin{theorem} \cite{BorVer:2018}
	Let $\kappa$ and $\lambda$ be different adjacency relations with $\kappa > \lambda$ on a digital image $X$. Then \[cat_{\kappa}(X) \leq cat_{\lambda}(X).\]
\end{theorem}

\begin{theorem}\cite{KaracaIs:2018}
	Let $(X,\kappa)$ be a digitally connected space such that $X \times X$ has $\kappa_{*}-$adjacency. Then \[cat_{\kappa}(X) \leq TC(X,\kappa) \leq cat_{\kappa_{*}}(X \times X).\]
\end{theorem}

\section{Topological Complexity of Maps In Digital Images}
\label{sec:2}
\begin{proposition}\label{p3}
	Let $f : (X_{1},\kappa_{1}) \rightarrow (Y_{1},\lambda_{1})$ and $g :(X_{2},\kappa_{2}) \rightarrow (Y_{2},\lambda_{2})$ be two digitally continuous digital maps. Let $f \times g : (X_{1} \times X_{2},\kappa_{*}) \rightarrow (Y_{1} \times Y_{2},\lambda_{*})$ be the digital product map. Then we have that
	\[genus_{\kappa_{*},\lambda_{*}}(f \times g) \leq genus_{\kappa_{1},\lambda_{1}}(f) + genus_{\kappa_{2},\lambda_{2}}(g).\] 
\end{proposition}
\begin{proof}
	We first consider the fibrations $f$ and $g$ on digital spaces and shall show the desired result. After that, we assume that $f$ and $g$ are map of digital images, do not have to be fibrations, and complete the proof. Let $genus_{\kappa_{1},\lambda_{1}}(f)=k$ and $genus_{\kappa_{2},\lambda_{2}}(g)=l$. We shall show that $genus_{\kappa_{*},\lambda_{*}}(f \times \nolinebreak g) \leq k+l$. Since $genus_{\kappa_{1},\lambda_{1}}(f) = k$, we may partition the digital image $Y_{1}$ into the $k$ subsets $U_{1}, U_{2}, ..., U_{k}$ such that for all $i = 1, ..., k$, there exist digitally continuous maps $s_{i} : (U_{i},\tau_{i}) \rightarrow (X_{1},\kappa_{1})$ and $f \circ s_{i}$ is an identity map on the digital images $(X_{1},\kappa_{1})$. Similarly, if $genus_{\kappa_{2},\lambda_{2}}(g) = l$, then we may partition the digital image $Y_{2}$ into $l$ subsets $V_{1}, V_{2}, ..., V_{l}$ such that there exist digitally continuous maps $t_{j} : (V_{j},\sigma_{j}) \rightarrow (X_{2},\kappa_{2})$, for all $j = 1, ..., l$, and $g \circ t_{j}$ is an identity map on the digital images $(Z_{2},\kappa_{2})$.
	Consider the digital map 
	\begin{eqnarray*}
		f \times g : X_{1} \times X_{2} \rightarrow (U_{1} \cup ... \cup U_{k}) \times (V_{1} \cup ... \cup V_{l}).
	\end{eqnarray*} 
	We rewrite this map in $2$ different ways:
	\begin{eqnarray}\label{eq1}
	f \times g : X_{1} \times X_{2} \rightarrow (U_{1} \times Y_{2}) \cup ... \cup (U_{k} \times Y_{2})
	\end{eqnarray}
	and
	\begin{eqnarray}\label{eq2}
	f \times g : X_{1} \times X_{2} \rightarrow (Y_{1} \times V_{1}) \cup ... \cup (Y_{1} \times V_{l}).
	\end{eqnarray} 
	Consider the equation (\ref{eq1}). Then there exists a digitally continuous map
	\begin{eqnarray*}
		w_{i} : (U_{i} \times Y_{2}) \rightarrow X_{1} \times X_{2}
	\end{eqnarray*}
	such that $f \circ w_{i}$ is the identity on $Y_{1} \times Y_{2}$. Similarly, for the equation (\ref{eq2}), we have a digitally continuous map
	\begin{eqnarray*}
		v_{j} : (Y_{1} \times V_{j}) \rightarrow X_{1} \times X_{2}
	\end{eqnarray*}
	such that $f \circ v_{j}$ is the identity on $Y_{1} \times Y_{2}$. Moreover, some of $U_{i} \times V_{j}$, for each $i$ and $j$, can be the same in the union of sets. So we conclude that $genus_{\kappa_{*},\lambda_{*}}(f \times g)$ must be less than or equal to $k+l$. When $f$ and $g$ are not fibrations in the digital sense, we use their digital fibrational substitutes to show that the desired inequality holds and this completes the proof.
\end{proof}

\begin{proposition} \label{p5}
	For a fibration $p : (E,\lambda_{2}) \rightarrow (B,\lambda_{3})$ of digital spaces, \[genus_{\lambda_{2},\lambda_{3}}(p) \leq cat_{\lambda_{3}}(B).\]
	Moreover, if $(E,\lambda_{2})$ is digitally contractible, then $genus_{\lambda_{2},\lambda_{3}}(p) = cat_{\lambda_{3}}(B)$.
\end{proposition}
\vspace*{-0.5cm}
\begin{proof}
	First, we shall show that $genus_{\lambda_{2},\lambda_{3}}(p) \leq cat_{\lambda_{3}}(B)$. Let $cat_{\lambda_{3}}(B) = k$. Then we have $k$ digital covering made by $k$ subset $\{U_{1}, U_{2}, ... , U_{k}\}$ of $B$, where each inclusion $U_{i} \rightarrow B$ for $i=1,...,k$ is digitally $\lambda_{3}-$null-homotopic in $B$. Assume $U \subseteq B$, where $U$ is one of the sets in the covering of $B$ and consider the following diagram for the positive integer $m$:
	\begin{displaymath}
	\xymatrix{
		U \times \{0\} \ar[r]^{c_{0}} \ar[d]_i &
		E \ar[d]^{p} \\
		U \times [0,m]_{\mathbb{Z}} \ar[r]^{H} \ar@{.>}[ur]^{G} & B,}
	\end{displaymath}
	where $i$ is the digital inclusion map. For any $t \in [0,m]_{\mathbb{Z}}$ and $b \in U$, $c_{0}$ is the digital constant map defined by $c_{0}(b,0) = e_{0}$, where $e_{0}$ is a chosen point in $p^{-1}(b_{0})$, for any basepoint $b_{0} \in B$. $H$ is a digital contracting homotopy between the digital constant map at the basepoint $b_{0}$ and the digital inclusion map $U \hookrightarrow B$. Using the digital homotopy lifting property, there is a digital map $G$ for which $G \circ i = c_{0}$ and $p \circ G = H$. It follows that \[p \circ G(x,m) = H(x,m) = id_{U}.\] If we take $G(x,m)$ as $G_{m}(x)$, then $G_{m}$ is a digital section of $p$ over $U$. Hence, we get the desired result.\\
	We now prove the second claim. Let $E$ be a digitally $\lambda_{2}-$contractible digital image. Let $genus_{\lambda_{2},\lambda_{3}}(p) = n$. Then there exists $A_{1}, A_{2}, ..., A_{n}$ of $B$ and, for each $A_{i}$, $s_{i} : A_{i} \to E$ is digitally continuous having that $p \circ s_{i} = 1_{A_{i}}$, where $1 \leq i \leq n$. Since $E$ is digitally contractible, $id_{E}$ is homotopic to the constant map on $E$ in digital images. Let us denote this digital homotopy with $H$. For any arbitrary $A_{i} \subset B$, we have the following construction: 
	\begin{displaymath} G : A_{i} \times [0,m]_{\mathbb{Z}} \stackrel{s_{i} \times id}{\longrightarrow} E \times [0,m]_{\mathbb{Z}} \stackrel{H}{\longrightarrow} E \stackrel{p}{\longrightarrow} B. 
	\end{displaymath}
	For all $a \in A$ and $t \in [0,m]_{\mathbb{Z}}$, conditions for being a digital homotopy of $G$ are held:
	\begin{eqnarray*}
		&& G(a,0) = p \circ H \circ (s \times id)(a,0) = p \circ H(s(a),0) = p \circ s(a) = id_{A}(a), \hspace*{0.1cm} \text{and} \\
		&& G(a,1) = p \circ H \circ (s \times id)(a,1) = p \circ H(s(a),1) = p \circ c_{s(a)}(a) = c_{p \circ s(a)}(a),
	\end{eqnarray*}
	where $c_{s(a)}$ is a constant digital map on $E$ at the point $s(a) \in E$ and $c_{p \circ s(a)}$ is a constant digital map on B at the point $p \circ s(a) \in B$. Moreover the digital maps $G|_{a} :[0,m]_{\mathbb{Z}} \to B$ and $G|_{t} : A \to B$ are digitally continuous. As a consequence, for all $1 \leq i \leq n$, the digital maps $A_{i} \to B$ is digitally nullhomotopic and thus we obtain $cat_{\lambda_{3}}(B) = n$.
\end{proof}

\quad By Proposition \ref{p5}, we immediately have the following:

\begin{proposition}
	For any connected digital image $(X,\kappa_{1})$ such that $X^{n}$ has $\kappa_{*}-$adjacency, we have that \[TC_{n}(X,\kappa_{1}) \leq cat_{\kappa_{*}}(X^{n}).\]
\end{proposition}	

\begin{proposition}
	Let $(X,\kappa_{1})$ be a connected digital image. Then we have
	\[cat_{\lambda_{*}}(X^{n-1}) \leq TC_{n}(X,\kappa_{1}),\]
	where $\lambda_{*}$ is an adjacency relation on $X^{n-1}$. 
\end{proposition}

\quad The proof can be modified in digital images with [Proposition 3.1,\cite{BasGonRudTamaki:2014}]. One can easily adapt the proof from topological spaces to digital images. The last two results give bounds for $TC_{n}$ using $cat$ in digital images.

\begin{corollary}\label{p2}
	Let $(X,\kappa_{1})$ be a connected digital image. Then
	\[cat_{\lambda_{*}}(X^{n-1}) \leq TC_{n}(X,\kappa_{1}) \leq cat_{\kappa_{*}}(X^{n}),\]
	where $\lambda_{*}$ and $\kappa_{*}$ is an adjacency relation on $X^{n-1}$ and $X^{n}$, respectively. 
\end{corollary}
\section{$\kappa-$Topological Groups In Digital Images}
\label{sec:3}
\quad We now have a new approach to compute $TC_{n}$ numbers of some of digital images. Our main equipment is the notion of topological groups in the digital sense.
\begin{definition}
	Let $(H,\kappa)$ be a digital image and $(H,\ast)$ be a group. Assume that the digital image $H \times H$ has a minimal adjacency relation for the cartesian product. If 
	\[\alpha : H \times H \to H \hspace*{1.0cm} \text{and} \hspace*{1.0cm} \beta : H \to H,\] 
	defined by $\alpha(y,z) = y \ast z$ and $\beta(y) = y^{-1}$, for all $y$, $z \in H$, respectively, are digitally continuous, then $(H,\kappa,\ast)$ is called a $\kappa-$topological group.
\end{definition}
 
\quad Notice that the hypothesis of minimality is necessary for $H \times H$. It is easy to see that $(\mathbb{Z},2,+)$ cannot be a $2-$topological group. Indeed, consider the digital map \newpage
\[\alpha: \mathbb{Z} \times \mathbb{Z} \to \mathbb{Z}\]
\hspace*{5.9cm}$(x,y) \longmapsto \alpha(x,y) = x+y$

and choose $8-$adjacency for $\mathbb{Z} \times \mathbb{Z}$. $(3,5)$ and $(4,6)$ are $8-$adjacent but $8$ and $10$ are not $2-$adjacent in $\mathbb{Z}$. It shows that $\alpha$ cannot be a digitally continuous map. But if we choose the minimal adjacency (\text{$4-$adjacency}) for $\mathbb{Z} \times \mathbb{Z}$, then $\alpha$ is a digitally continuous map. Hence, $(\mathbb{Z},2,+)$ is a $2-$topological group. Inversely, we note a difference between topological spaces and digital images: In topological spaces, $(\mathbb{R}^{\ast},\tau_{s},\cdot)$ is a topological group, where $\mathbb{R}^{\ast}$ denotes the set $\mathbb{R} - \{0\}$. This does not give a response in digital images. Consider the triple $(\mathbb{Z}^{\ast},2,\cdot)$, where $\mathbb{Z}^{\ast} = \mathbb{Z} - \{0\}$. Even $\mathbb{Z}^{\ast}$ is not a monoid under $\cdot$ because the inverse of $2$ does not exists. As a result $(\mathbb{Z}^{\ast},2,\cdot)$ does not have a $2-$topological group structure.

\quad We begin with a trivial example of $\kappa-$topological groups. We give another example with a different construction.
\begin{example} \label{e1}
	Let $G = \{-1,1\} \subset \mathbb{Z}$ be a digital image. Then $G$ is a group under $\cdot$ in $\mathbb{Z}$. Consider the digital maps \[\alpha : G \times G \longrightarrow G\]
	\hspace*{5.2cm} $(x,y) \longmapsto \alpha(x,y) = x \cdot y$\\
	and \[\beta : G \longrightarrow G\] \hspace*{5.7cm} $x \longmapsto \beta(x) = x$.
	
	In the domains of $\alpha$ and $\beta$, there does not exist any adjacent pair of points. It means that $\alpha$ and $\beta$ are trivially digitally continuous. Consequently, $(G,2,\cdot)$ is a $2-$topological group.
\end{example}

\begin{example}\label{e2}
	Given an integer $m$, let $H = [m,m+1]_{\mathbb{Z}} \subset \mathbb{Z}$. For having a group construction on $H$, take a binary operation $\ast$ such that for all $a,b \in H$,
	\[\ast (a,b) = \begin{cases}
	m, & a=b \\ m+1, & a \neq b.
	\end{cases}\]
	The digital map
	\[ \alpha : H \times H \rightarrow H\]
	\hspace*{5.7cm} $(a,b) \longmapsto \alpha(a,b) = a \ast b$
	
	is digitally continuous because of the fact that $a \ast b = m$ or $m+1$. In addition, another digital map
	\[\beta : H \rightarrow H\]
	\hspace*{5.9cm} $m \longmapsto \beta(m) = m$\\
	\hspace*{5.3cm} $m+1 \longmapsto \beta (m+1) = m+1$
	
	is clearly digitally continuous. It shows that $(H,2,\ast)$ is a $2-$topological group.
\end{example}	

\begin{theorem}
	Let $m$ be any integer. Then there is no $2-$topological group structure on the digital interval $[m,m+p-1]_{\mathbb{Z}}$, for all prime $p \geq 3$.
\end{theorem}

\begin{proof}
	Let $p = 3$. Assume that $[m,m+2]_{\mathbb{Z}}$ has  $2-$topological group structure with any group operation $\ast$ and the $2-$adjacency relation. It means that $([m,m+2]_{\mathbb{Z}},\ast)$ is a group in the algebraic sense. Moreover, the digital maps 
	\[\alpha : [m,m+2]_{\mathbb{Z}} \times [m,m+2]_{\mathbb{Z}} \rightarrow [m,m+2]_{\mathbb{Z}} \hspace*{0.2cm} \text{and} \hspace*{0.2cm} \beta : [m,m+2]_{\mathbb{Z}} \rightarrow [m,m+2]_{\mathbb{Z}}\] 
	are digitally continuous. Then there are three cases for identity element of the group: $e_{[m,m+2]_{\mathbb{Z}}}$ is equal to only one of $m, m+1$ and $m+2$. Assume that $m$ is the identity element. Since $3$ is prime, every group of $3$ elements is the cyclic group of order $3$. Moreover, the set $\{m,m+1,m+2\}$ is an abelian group and every element different from the identity is a generator. This gives us the following properties:
	\begin{eqnarray*}
		&&(m+2)\ast(m+2) = (m+1),\\
		&&(m+1)\ast(m+2) = (m+2)\ast(m+1) = e_{[m,m+2]_{\mathbb{Z}}},\\
		&&(m+1)^{-1} = m+2 \hspace*{0.3cm} \text{and} \hspace*{0.3cm} (m+2)^{-1} = m+1.\\
	\end{eqnarray*}
	If $e_{[m,m+2]_{\mathbb{Z}}} = m$, then we find $\beta(m) = m$ and $\beta(m+1) = m+2$. This means that $\beta$ is not digitally continuous. This is a contradiction.
	Now consider the second case. In other words, let $m+1$ be an identity element of the group. Then $\alpha$ is not digitally continuous because we get
	\begin{eqnarray*}
		\alpha(m,m+1) = m \hspace*{0.3cm} \text{and} \hspace*{0.3cm} \alpha(m+1,m+2) = m+2.
	\end{eqnarray*}
	This is again contradiction. Consider the third case, i.e., $m+2$ is the identity element of the group. The case is symmetric to the case $e_{[m,m+2]_{\mathbb{Z}}} = m$ since the map that swaps $m$ and $m+2$ is an isomorphism of digital images. As a consequence, $([m,m+2]_{\mathbb{Z}},2,\ast)$ cannot be a $2-$topological group. If $p$ is a prime with $p > 3$, then the idea can be generalized because we have two elements, namely the endpoints $m$ and $m+p-1$, that have only one adjacent element, while, by the symmetry induced by the group action, each element have precisely two adjacent elements. 
\end{proof}

\begin{proposition}\label{p1}
	Let $(H,\kappa,\ast)$ and $(H,^{'}\lambda,\circ)$ be a $\kappa-$topological group and a $\lambda-$topological group, respectively. Then their cartesian product $H \times H^{'}$ is also a $\kappa_{*}$ topological group, where $\kappa_{*}$ is a minimum adjacency relation for the image $H \times H.^{'}$
\end{proposition}

\begin{proof}
	Let $H$ be a $\kappa-$topological group. Then the digital maps
	\[\alpha_{1} : H \times H \to H \hspace*{1.0cm} \text{and} \hspace*{1.0cm} \beta_{1} : H \to H,\] 
	defined by $\alpha_{1}(y_{1},z_{1}) = y_{1} \ast z_{1}$ and $\beta_{1}(y_{1}) = y_{1}^{-1}$ for all $y_{1}$, $z_{1} \in H$, respectively, are digitally continuous. Similarly, for the $\lambda-$topological group $H,^{'}$ we have that the digital maps
	\[\alpha_{2} : H^{'} \times H^{'} \to H^{'} \hspace*{1.0cm} \text{and} \hspace*{1.0cm} \beta_{2} : H^{'} \to H,^{'}\] 
	defined by $\alpha_{2}(y_{2},z_{2}) = y_{2} \circ z_{2}$ and $\beta_{2}(y_{2}) = y_{2}^{-1}$, for all $y_{2}$, $z_{2} \in H,^{'}$ are digitally continuous. Define a digital map
	\[\alpha = \alpha_{1} \times \alpha_{2} : H \times H \times H^{'} \times H^{'} \to H \times H^{'} \hspace*{1.0cm}\]
	\hspace*{4.7cm} $((y_{1},z_{1}),(y_{2},z_{2})) \longmapsto (y_{1} \ast z_{1},y_{2} \circ z_{2})$.
	
	We shall show that $\alpha$ is a digitally continuous map. The product of digitally continuous maps is digitally continuous with a minimal adjacency relation. Let $((y_{1},z_{1}),(y_{2},z_{2}))$ and $((y_{1}{'},z_{1}{'}),(y_{2}{'},z_{2}{'}))$ be digitally connected points. Then $(y_{1},z_{1})$ is digitally connected with $(y_{1}{'},z_{1}{'})$ and $(y_{2},z_{2})$ is digitally connected with $(y_{2}{'},z_{2}{'})$. Since $\alpha_{1}$ is digitally continuous, $y_{1} \ast z_{1}$ is digitally connected with $y_{1}{'} \ast z_{1}{'}$. Similarly, for the digital continuity of $\alpha_{2}$, we have that $y_{2} \circ z_{2}$ is digitally connected with $y_{2}{'} \circ z_{2}{'}$. Cartesian product adjacency gives that $\alpha$ is digitally continuous. In order to satisfy the other condition, we define the digital map
	\[\beta = \beta_{1} \times \beta_{2} : H \times H^{'} \to H \times H^{'} \hspace*{1.0cm}\]
	\hspace*{5.4cm} $(y_{1},z_{1}) \longmapsto (y_{1}^{-1},z_{1}^{-1})$.
	
	Let $(y_{1},z_{1})$ and $(y_{2},z_{2})$ be digitally connected points for the cartesian product. Then we have that $y_{1}$ is digitally connected with $y_{2}$ and $z_{1}$ is digitally connected with $z_{2}$. Since $\beta_{1}$ and $\beta_{2}$ are digitally continous, we obtain that $y_{1}^{-1}$ is digitally connected with $y_{2}^{-1}$. Similary, for the digital continuity of $\beta_{2}$, we obtain that $z_{1}^{-1}$ is digitally connected with $z_{2}^{-1}$. Using the definition of the adjacency for the cartesian product, we conclude that $\beta$ is digitally continuous. This gives the required result. 
\end{proof}

\begin{definition}
	Let $(H,\kappa,\ast)$ and $(H{'},\lambda,\circ)$ be a $\kappa-$topological group and a $\lambda-$topological group, respectively. Then a digital map $\gamma : (H,\kappa,\ast) \rightarrow (H{'},\lambda,\circ)$ is a $(\kappa,\lambda)-$homomorphism between $\kappa-$topological group and $\lambda-$topological group if $\gamma$ is both digitally continuous and group homomorphism. A $(\kappa,\lambda)-$isomorphism between  $\kappa-$topological group and $\lambda-$topological group is both digital isomorphism and group homomorphism.
\end{definition}

\begin{example}
	It is easy to see that $(\mathbb{Z}^{2},4,+)$ is $4-$topological group by Proposition \ref{p1}. Consider the digital projection map
	\[\alpha : (\mathbb{Z}^{2},4,+) \longrightarrow (\mathbb{Z},2,+)\]
	\hspace*{5.0cm} $(m,n) \longmapsto m$.
	
	We prove that $\alpha$ is a $(4,2)-$homomorphism in the sense of topological groups but it is not a $(4,2)-$topological group isomorphism. $\alpha$ is a digitally continuous map because $m_{1}$ and $m_{2}$ are $2-$connected whenever $(m_{1},n_{1})$ and $(m_{2},n_{2})$ are $4-$adjacent points in $\mathbb{Z}^{2}$. Using the fact that the projection maps associated with a product of groups are always group isomorphisms, we have that $\alpha$ is a group homomorphism. Hence, we prove that $\alpha$ is a $(4,2)-$topological group homomorphism. On the other hand, the projection maps are not injective. Finally, we show that $\alpha$ is not a $(4,2)-$topological group isomorphism.
\end{example}

\quad Note that the digital isomorphism of two topological groups is stronger than simply requiring a digitally continuous group isomorphism. The inverse of the digital function must also be digitally continuous. The next example shows that two topological groups in digital images are not digitally isomorphic in the sense of topological groups whenever they are isomorphic as ordinary groups.

\begin{example}
	Consider the $2-$topological group $(G,2,\cdot)$ given in Example \ref{e1}. Let $(H,2,\ast)$ be another $2-$topological group for which $H = [8,9]_{\mathbb{Z}} \subset \mathbb{Z}$ and $\ast$ is the same group operation given in Example \ref{e2}. Then the digital map \linebreak $f : (G,2,\cdot) \rightarrow (H,2,\ast)$, defined by $f(1) = 8$ and $f(-1) = 9$, is an isomorphism of algebraic groups but not a $(2,2)-$isomorphism of topological groups. It is clear that $f$ is bijective. Further, $f$ preserves the group operation:
	\begin{eqnarray*}
		&&f(1 \cdot 1) = f(1) = 8 = 8 \ast 8 = f(1) \ast f(1) \\
		&&f(1 \cdot -1) = f(-1) = 9 = 8 \ast 9 = f(1) \ast f(-1) \\
		&&f(-1 \cdot 1) = f(-1) = 9 = 9 \ast 8 = f(-1) \ast f(1) \\
		&&f(-1 \cdot -1) = f(1) = 8 = 9 \ast 9 = f(-1) \ast f(-1). 
	\end{eqnarray*}
	There is no adjacent points in $G$. So, $f$ is digitally continuous. Contrarily, $8$ and $9$ are $2-$adjacent but $f^{-1}(8) = 1$ and $f^{-1}(9) = -1$ are not $2-$adjacent. Hence, the inverse of $f$ is not digitally continuous.
\end{example}

\begin{theorem}
	If $H$ is a subgroup of a $\kappa-$topological group $(G,\kappa,\ast)$, then $(H,\kappa,\ast)$ is a $\kappa-$topological group.
\end{theorem}

\begin{proof}
	Suppose that $(G,\kappa,\ast)$ is a topological group. Then 
	\[\alpha : G \times G \rightarrow G \hspace*{1.0cm} \text{and} \hspace*{1.0cm} \beta : G \rightarrow G \]
	\[\hspace*{1.1cm}(x,y) \longmapsto x \ast y \hspace*{2.7cm} x \longmapsto x^{-1}\]
	are digitally continuous. To show that the digital maps 
	\[\alpha_{1} : H \times H \rightarrow H \hspace*{1.0cm} \text{and} \hspace*{1.0cm} \beta_{1} : H \rightarrow H \]
	\[\hspace*{1.1cm}(a,b) \longmapsto a \ast b \hspace*{2.7cm} a \longmapsto a^{-1}\]
	are continuous, it is enough to demonstrate that $H \leq G$ and $\alpha$ and $\beta$ are digitally continuous. Indeed, for two adjacent points in $H \times H$, they are also adjacent in $G \times G$ and their images are adjacent in $G$. The adjacency relation in $H$ is the same for $G$. Therefore, their images are also adjacent in $H$. It shows that $\alpha_{1}$ is digitally continuous. Similarly, $\beta_{1}$ is digitally continuous. The continuity of $\alpha_{1}$ and $\beta_{1}$ gives the desired result. 
\end{proof} 
\section{Some Results For The Digital Higher Topological Complexity}
\label{sec:4}
\begin{theorem}\label{t1}
	Let $(H,\kappa,\cdot)$ be a $\kappa-$topological group such that $(H,\kappa)$ is digitally connected and $n > 1$. Then \[TC_{n}(H,\kappa) = cat_{\kappa_{\ast}}(H^{n-1}),\]
	where $\kappa_{\ast}$ is an adjacency relation for $H^{n-1}$.
\end{theorem}

\begin{proof}
	By Proposition \ref{p2}, it is enough to show that $TC_{n}(H,\kappa) \leq r$ when $r$ equals $cat_{\kappa}(H^{n-1})$. Suppose that $\{M_{1}, M_{2}, ..., M_{r}\}$ is a covering of $H^{n-1}$, where all $M_{i}$'s are digitally contractible in $H^{n-1}$, for all $i = 1, ... ,r$. In other saying, $M_{i}$ contracts to an element $(h_{1},h_{2},...,h_{n-1})$ in $H^{n-1}$ for each $i$. Since $H$ is a $\kappa-$topological group, it has the identity element $e_{H}$. Let $(e_{H}, e_{H}, ...,e_{H})$ be denoted by $e_{H}^{(n-1)}$. Each contracting homotopy can be extended in $H^{n-1}$ such that \linebreak $(h_{1},h_{2},...,h_{n-1}) = e_{H}^{(n-1)}$ for all $i = 1,...,r$ because $H$ is $\kappa-$connected. Now, we define
	\begin{eqnarray*}
		N_{i} = \{(h,hm_{1},...,hm_{n-1}) : (m_{1},...,m_{n-1}) \in M_{i}, \hspace*{0.2cm} h \in H\}.
	\end{eqnarray*}
	We shall show that $e_{H}^{(n)}$ admits a digitally continuous section over each $N_{i}$. Let $m = (m_{1},...,m_{n-1})$. The digital contractibility of $M_{i}$ gives a digital path $\alpha_{m}$ and this path joins $e_{H}^{(n)}$ to each $m \in M_{i} \subset H^{n-1}$. We define a new digital path $\alpha_{m}{'}$ from $e_{H}^{(n)}$ to $(e_{H},m_{1},...,m_{n-1})$ in $N_{i}$. Then for any $h \in H$, $g\alpha_{m}{'}$ is a digital path in $H^{n}$ from $(h,h,...,h) = he_{H}^{(n)}$ to $(h,hm_{1},...,hm_{n-1})$. Finally, we define the digitally continuous map
	\begin{eqnarray*}
		s_{i} : N_{i} \rightarrow H^{J_{n}}
	\end{eqnarray*}
	as $s_{i}(h,hm_{1},...,hm_{n-1})$ is the $j-$th element of $h\alpha_{m}{'}$ on the $j-$th digital interval of $J_{n}$. Hence, we get $H^{n} = N_{1} \cup ... \cup N_{r}$. If we take $(c_{1},...,c_{n}) \in H^{n}$ and $h = c_{1}$, then $m_{i} = h^{-1}c_{i}$. So, there exists $j$ such that $(m_{1},...,m_{n}) \in M_{j}$. This means that $(c_{1},...,c_{n}) \in N_{j}$. As a result, $TC_{n}(H,\kappa) \leq r$.
\end{proof}

\begin{example}\label{e3}
	Consider the digital image $H$ given in Example \ref{e4}. $(H,4,\circ)$ is a $4-$topological group, where $\circ$ is a group operation:
	\begin{table}[h!]
		\centering
		\begin{tabular}{c|c|c|c|c|c|c|c|c}
			$\circ$ & $a$ & $b$ & $c$ & $d$ & $e$ & $f$ & $g$ & $h$ \\
			\hline
			a & h & a & b & c & d & e & f & g\\
			\hline
			b & a & b & c & d & e & f & g & h\\
			\hline
			c & b & c & d & e & f & g & h & a\\
			\hline
			d & c & d & e & f & g & h & a & b\\
			\hline
			e & d & e & f & g & h & a & b & c\\
			\hline
			f & e & f & g & h & a & b & c & d\\
			\hline
			g & f & g & h & a & b & c & d & e\\
			\hline
			h & g & h & a & b & c & d & e & f.
		\end{tabular}
		\vspace*{0.2cm}\caption{The group operation $\circ$ for $H$.}
		\label{tab:table4}
	\end{table}
	
	Note that $H$ is a cyclic group where $b$ is the identity, and $a$ is a generator. $H$ is not $4-$contractible digital image, so it is true that $TC_{n}(H,4) = 1$ only when $n=1$. To compute $TC_{2}(H,4)$, we use Theorem \ref{t1}. By Example \ref{e4}, we obtain $cat_{4}(H) = 2$. As a result, we get $TC_{2}(H,4) = 2$.
\end{example}

\begin{corollary} \label{c1}
	Let $(H,\kappa,\ast)$ be a $\kappa-$topological group with $(H,\kappa)$ is digitally connected. Then for $n>2$,
	\[TC_{n}(H,\kappa) - TC_{n-1}(H,\kappa) \leq cat_{\kappa}(H).\]
\end{corollary} 

\begin{proof}
	By Proposition \ref{p4}, we have $TC_{n-1}(H,\kappa) = genus_{\kappa,\kappa_{*}}(d_{n-1})$, where \[d_{n-1} : (H,\kappa) \rightarrow (H^{n-1},\kappa_{1})\] is a diagonal map of digital images with the adjacency relation $\kappa_{1}$ for $H^{n-1}$. Furthermore, Theorem \ref{t1} allows us that $cat_{\kappa}(H) = TC_{2}(H,\kappa)$. Therefore, we get $TC_{2}(H,\kappa) = genus_{\kappa,\kappa_{2}}(d_{2})$, where $d_{2}$ is also a diagonal map with the adjacency relation $\kappa_{2}$ for $H^{2}$. Proposition \ref{p3} admits that \[genus_{\kappa_{2},\kappa_{3}}(d_{n-1} \times d_{2}) \leq genus_{\kappa,\kappa_{1}}(d_{n-1}) + genus_{\kappa,\kappa_{2}}(d_{2})\] with the adjacency relation $\kappa_{3}$ for $H^{n+1}$. Considering that the cartesian product of diagonal maps is $d_{n-1} \times d_{2} : (H^{2},\kappa_{2}) \rightarrow (H^{n+1},\kappa_{3})$, we conclude that
	\[TC_{n}(H,\kappa) - TC_{n-1}(H,\kappa) \leq cat_{\kappa}(H).\] 
\end{proof}

\begin{example}
	From Example \ref{e3}, we have $TC_{2}(H,4) = cat_{4}(H) = 2$. Corollary \ref{c1} gives an idea for the upper bound of $TC_{3}(H,4)$ without having to work in $\mathbb{Z}^{4}$. It indicates that $TC_{3}(H,4) - TC_{2}(H,4) \leq cat_{4}(H)$ and hence $TC_{3}(H,4) \leq 4$.
\end{example}
\section{Conclusion}\label{sec6}
\label{sec:5}
\quad We first considered a relation between the Lusternik-Schnirelmann theory and the higher topological complexity more conceretely in digital images. Second, our task is to include  $\kappa-$topological groups in our study. While doing theoretical modeling, we also observe examples of digital images that might be useful in later works. We try to get the properties in terms of the digital higher topological complexity. Some theoretical infrastructure needs to be established before accessing the applications of motion planning algorithms in digital images. So, these results are valuable in our opinion. We wish to progress to the wide application area of motion planning algorithms by proceeding step by step. We intend to make an impact on at least one application area for the future works. For example, in computer games, virtual characters have to use motion planning algorithms to determine their direction and find a way between two locations in the virtual environment. In addition to this, we can encounter motion planning problem in almost every aspect of our life such as military simulations, probability and economics, artificial intelligence, urban design, robot-assisted surgery and the study of biomolecules.

\acknowledgment{This work was partially supported by Research Fund of the Ege University (Project Number: FDK-2020-21123). In addition, the first author is granted as fellowship by the Scientific and Technological Research Council of Turkey TUBITAK-2211-A.}

\end{document}